\documentclass[12pt]{amsart}

\usepackage{amssymb,epsfig,mathrsfs}
\usepackage{fullpage}
\usepackage{graphicx}
\usepackage{mathtools}
\usepackage{xspace}  
\usepackage[colorlinks=true, urlcolor=blue, citecolor=blue, linkcolor = blue]{hyperref}

\usepackage{tikz-cd}
\usepackage{tikz-qtree,tikz-qtree-compat, forest}
\usetikzlibrary{shapes.geometric}

\def\Z{{  \mathbb Z }}
\def\Q{{\mathbb Q}}

\def\F{{\mathbb F}}

\def\IPAD{{\mathit{IPAD}}}

\def\magma{{\sc Magma}}
\def\pari{{\sc PARI/GP}}

\def\la{\langle}
\def\ra{\rangle}

\def\ab{{\mathrm{ab}}}

\def\Gal{\mathrm{Gal}}
\def\Cl{\mathrm{Cl}}

\newtheorem{theorem}{Theorem}[section]
\newtheorem{conjecture}[theorem]{Conjecture}
\newtheorem{thm-fr}{Th\'eor\`eme}[section]
\newtheorem{lem-fr}{Lemme}[section]
\newtheorem{cor-fr}{Corollaire}[section]
\newtheorem{lemma}[theorem]{Lemma}

\newtheorem{prop-fr}[theorem]{Proposition}

\theoremstyle{remark}

\newtheorem{remark}[theorem]{Remark}
\newtheorem{rem-fr}[theorem]{Remarque}

\newtheorem{definition}[theorem]{Definition}
\newtheorem{def-fr}[theorem]{D\'efinition}

\begin{document}

\title{Heuristics for $2$-class Towers of Cyclic Cubic Fields}

\begin{abstract}
We consider the Galois group $G_2(K)$ of the maximal unramified $2$-extension of $K$ where  $K/\mathbb{Q}$ is cyclic of degree $3$. We also consider the group $G^+_2(K)$ where ramification is allowed at infinity. In the spirit of the Cohen-Lenstra heuristics, we identify certain types of pro-$2$ group as the natural spaces where $G_2(K)$ and $G^+_2(K)$ live when the $2$-class group of $K$ is $2$-generated. While we do not have a theoretical scheme for assigning probabilities, we present data and make some observations and conjectures about the distribution of such groups.
\end{abstract}

\author{Nigel Boston}
\address{Department of Mathematics, University of Wisconsin - Madison,  
480 Lincoln Drive, Madison, WI 53706, USA}
\email{boston@math.wisc.edu}

\author{Michael R. Bush}
\address{Department of Mathematics,
Washington and Lee University,
Lexington, Virginia 24450, USA}
\email{bushm@wlu.edu}


\maketitle

\section{Introduction}

In the 1980s, Cohen and Lenstra~\cite{CL2} gave a theoretical framework for the variation of class groups of quadratic fields.  
The Cohen-Lenstra idea is twofold: the first part is to identify, in any relevant number-theoretical situation, the correct collection of groups which can arise as the groups of number-theoretical interest; the second part is to define a natural measure or probability distribution on this collection.
The heuristic then is that the probability attached to the group in the identified collection is the same as the frequency of occurrence as a group of number-theoretical interest.

In \cite{BBH} and \cite{BBH2}, we initiated the study of a natural non-abelian extension of Cohen and Lenstra's work. 
For a number field $K$ and rational prime $p$, we considered the Galois group $G_p(K)$ of the maximal unramified $p$-extension of $K$. Note that the maximal abelian quotient $G_p(K)^{\ab}$ is isomorphic to the $p$-class group of $K$ by class field theory.  We call $G_p(K)$ the $p$-class tower group of $K$, since that is how it first arose in the 1930s in the work of Artin, Hasse, Furtw\"angler and others.  In \cite{BBH}, together with Farshid Hajir, we treated the case of imaginary quadratic fields and in \cite{BBH2} the case of real quadratic fields, when $p$ is an odd prime.  The content of each of these papers included a) an identification of the ``right'' collection of groups, b) an investigation of an associated measure giving the frequency of groups within that collection, and c) a numerical study of $p$-class tower groups of both types of quadratic field to test the conjecture developed using a) and b).  

In this work, we treat cyclic cubic fields in a similar manner for the most interesting case when $p = 2$.  The situation when $p > 3$ gives rise to a theory very much like that elucidated in our earlier papers, and one that fits naturally into the more general framework recently developed by Liu, Wood, and Zureick-Brown \cite{LWZ}. The case of $p = 2$ is as yet uncharted and introduces some new and unexpected considerations, related to the fact that the ground field $\Q$ contains a primitive $p$th root of unity \cite{Ga}, \cite{M08}. Here even the study of frequency of $p$-class groups is relatively untouched, with some first steps taken by Malle and Adam \cite{MA}, building on \cite{M08} and \cite{M10}. The case of $p=3$ is quite different, and related to extensions of genus theory \cite{Kl}, \cite{MAD}. 

The organization of the paper is as follows. In Section~2, we introduce our main objects of interest, cyclic cubic fields and their narrow and wide $2$-class tower groups. This leads us to focus on a subclass of $2$-groups, called valid, and we describe how these are represented by the nodes of a tree. In Section~3, we associate to each valid group a $3$-tuple of abelian groups, called its IPAD. The IPADs of tower groups are readily calculated and can then be used to identify the groups or at least narrow down the possibilities. In Section~4, we use embedding theory and Galois cohomology to show that the narrow $2$-class tower groups lie in a certain class of groups, which we call $2$-special, and in Section~5 we likewise identify a class of groups, called $2$-select, that contains the wide $2$-class tower groups.

We conjecture that every $2$-special group does arise as the narrow tower group of some cyclic cubic field and that every $2$-select group arises as the wide tower group of some such field. Ultimately we would like to know the joint distribution, but there appear to be some as yet unknown subtleties as to the support of this distribution. After summarizing previous work on the distribution for the wide $2$-class groups, in Section~7 we present data on the narrow and wide IPADs of the first $500,000$ cyclic cubic fields with $2$-generated $2$-class group, ordered by discriminant. This is then used to substantiate the claims at the start of this paragraph.

Although we cannot obtain general formulae for how frequently a given group arises as a narrow or wide $2$-class tower group, the data give us explicit, albeit somewhat mysterious, values for the joint distribution of the $2$-class $3$ quotients of the narrow and wide tower groups. We can also pin down the joint distribution induced on the abelianizations (the narrow and wide $2$-class groups). 

One consequence is an explanation of Rubinstein-Salzedo's observation that cyclic cubic fields with $2$-class group of order $4$ have $2$-class tower group of order $4$ about $53.7\%$ of the time and of order $8$ about $46.3\%$ of the time. At the same time we explain why this is still consistent with the conjecture that unramified $(\Z/2)^2$-extensions of cyclic cubic fields are as likely as not to extend to unramified $Q_8$-extensions.

Lastly, in Section~8, we consider the possibility of infinite narrow $2$-class towers. By random search, filtering out presentations that yield finite or $2$-adic analytic groups, we are led to a very special family of groups. Moreover, the corresponding wide $2$-class tower group then turns out to be a certain group of order $2048$.


\section{Preliminaries}

Throughout this paper, $K$ will denote a cyclic cubic field with $\Delta = \Gal(K/\Q) = \langle \sigma \rangle$ where $\sigma$ is an automorphism of order $3$.  
Let $\Cl_2(K)$ denote the $2$-class group of $K$ and $G_2(K) = \Gal(K^{nr,2}/K)$ where $K^{nr,2}/K$ is the maximal unramified $2$-extension of $K$. We will call $G_2(K)$ the {\em $2$-class tower group of $K$}. By class field theory, $\Cl_2(K)$ is isomorphic to the abelianization $G_2(K)/[G_2(K),G_2(K)]$. If we allow ramification at infinity, then we get the narrow analogs of these objects. In particular, we will let $\Cl^+_2(K)$ denote the narrow $2$-class group and call $G^+_2(K)$ the {\em narrow $2$-class tower group of $K$}. To distinguish, $Cl_2(K)$ and $G_2(K)$ are often called {\em wide}.

By the Schur-Zassenhaus Theorem, the Galois group $\Delta$ acts by conjugation on $G_2(K)$ and $G^+_2(K)$. Moreover $G_2(K)$ is the quotient of $G^+_2(K)$ by the closure of complex conjugation under this action.
Under the isomorphism to the corresponding wide and narrow $2$-class groups, the induced action on the abelianizations translates into the usual Galois action on ideal classes. Let $V$ denote the unique irreducible $\mathbb{F}_2[\Delta]$-module of dimension $2$ over $\mathbb{F}_2$.

\begin{lemma} As $\mathbb{F}_2[\Delta]$-modules, $\Cl_2(K)/\Cl_2(K)^2$ and $\Cl^+_2(K)/\Cl^+_2(K)^2$ are each isomorphic to $V^d$ for some non-negative integer $d$. It follows that the generator rank $g$ of $\Cl_2(K)$ (which equals that of $G_2(K)$) is even, namely $2d$, and equal to that of $\Cl^+_2(K)$ (and so of $G^+_2(K)$).
\end{lemma}

\begin{proof} Since $\gcd(2,3) = 1$, no ideal class $I$ can be fixed under this action since the product $I \sigma(I) \sigma^2(I)$ is trivial. By Maschke's Theorem, it follows that the $\mathbb{F}_2[\Delta]$-module $\Cl_2(K)/\Cl_2(K)^2$ decomposes as a direct sum of copies of $V$. In particular, we see that the generator rank is always even since $V$ has dimension $2$. Similar arguments apply to $\Cl^+_2(K)$ and  $G^+_2(K)$. By~\cite{AF}, the generator ranks of $\Cl_2(K)$ and $\Cl^+_2(K)$ differ by at most $1$ and so they are always equal in our situation.
\end{proof}

The group $G_2(K)$ is a finitely presented pro-$2$ group and $\Cl_2(K)$ is a finite abelian $2$-group. Before proceeding further we recall some notions and terminology from group theory that will be useful.  For a pro-$2$ group $G$, recall that $d(G)=\dim_{\F_2} H^1(G,\F_2)$ and $r(G)=\dim_{\F_2} H^2(G,\F_2)$ are its minimal number of (topological) generators and relations, respectively.  Setting $g = d(G)$, we can view $G$ as a quotient of the free pro-$2$ group $F$ on $g$ generators $x_1,...,x_g$.

The $2$-central series of a pro-$2$ group $G$ is the series of closed subgroups defined by:
\[ P_0(G) = G, \ \ P_{n+1}(G) = P_n(G)^2 [G,P_n(G)]. \] Note that $P_1(G)$ is the Frattini subgroup, $\Phi(G)$.
If $P_n(G) = 1$ for some $n$ and $n$ is minimal, then we say $G$ has $2$-class $n$. Finite $2$-groups always have finite $2$-class. If $G$ is a finitely generated pro-$2$ group, then $Q_n(G): = G/P_n(G)$ is a finite $2$-group which we call the maximal $2$-class $n$ quotient of $G$. 
We say that $G$ is a descendant of $Q_n(G)$ and call $Q_{n+1}(G)$ an immediate descendant or child of $Q_n(G)$ (if not isomorphic to it). If we connect groups that are immediate descendants with an edge, then the finite $g$-generated $2$-groups form a tree with root the elementary abelian $2$-group $(\Z/2)^g$.
There is an algorithm due to O'Brien~\cite{O}, implemented in \magma~\cite{magma}, that yields all the children of any given $p$-group, thus allowing one to compute successive levels of this tree.

In later sections, we will restrict attention to the situation where $G = G_2(K)$ has generator rank $g = d(G) = 2$. In this case, $Q_1(G) \cong (\Z/2)^2$. For brevity's sake, we use the standard notation $[2,2]$ for this abelian group. By O'Brien's algorithm, $[2,2]$ has $7$ children, namely $[2,4], D_8, 
Q_8, [4,4]$, two other groups of order $16$, and $Q_2(F)$ of order $32$. 

\begin{definition} A $2$-generator (pro)-$2$-group with an automorphism of order $3$ will be called {\em valid}. 
\end{definition}

All the Galois groups we are considering are valid, thanks to the $\Delta$-action. If $G$ is valid, then since the subgroups in its $2$-central series are characteristic, the quotient $Q_n(G)$  is itself valid for all $n \geq 1$.  Of the $7$ children of $[2,2]$, only $Q_8, [4,4],$ and $Q_2(F)$ are valid. Of these, $Q_8$ has no children, whereas the only valid descendants of $[4,4]$ are the abelian groups $[2^n,2^n]$. This can be verified by first showing that the $2$-covering group for $[2^n,2^n]$ (as defined in~\cite{O}) is the group $G^*$ of order $2^{2n + 3}$ with presentation
\[  G^* = \langle x, y \mid x^{2^{n+1}}, y^{2^{n+1}}, (x,y)^2, (x,(x,y)), (y,(x,y)) \rangle \]
where $(x,y) = x^{-1} y^{-1} x y$. By considering suitable quotients of this group, one can verify that $[2^n,2^n]$ has $4$ children for all $n \geq 2$ and that only the child $[2^{n+1},2^{n+1}]$ is valid. 

It follows that most of the groups we encounter later are descendants of $Q_2(F)$. This has $93$ children, but only $11$ of them are valid. 

Of these $11$, $3$ are childless; $1$ has just $2$ valid children, both childless; $1$ has $3$ valid children and a very straightforward subtree of descendants; 3 have $7$ valid children; $1$ has $11$ valid children; $1$ has $17$ valid children; and $1$ has $149$ valid children.

Figure 1 depicts this tree. Note that $\la n,i \ra$ is shorthand for SmallGroup(n,i) in the standard Magma database~\cite{magma}. 

\begin{figure}
\caption{Tree of valid $2$-groups of $2$-class at most $4$.}
\begin{center}
\tikzstyle{every node}=[draw,shape=rectangle, scale=0.6]
\begin{forest}
for tree={s sep*=0.8, l sep*=2.6, inner sep=4pt}
[{$[2,2]$}
	[$Q_8$]
	[,phantom]
	[{$[4,4]$} [{$[8,8]$} [{$[16,16]$} ]]]
	[{$Q_2(F)$} [{$\la 64,19 \ra $}]  [,phantom] [{$\la 128,5 \ra $} [{$3$ children}]] [{$\la 128,36 \ra $} [{$11$ children}]] [{$\la 128,40 \ra $}] [{$\la 128,41 \ra $}] [{$\la 256,2 \ra $} [{$7$ children}]] [{$\la 256,35 \ra $} [{$7$ children}]] [{$\la 256,36 \ra $} [{$2$ children}]] [{$\la 256,38 \ra $} [{$7$ children}]] [{$\la 512,3 \ra $} [{$17$ children}]] [{$Q_3(F)$} [{$149$ children}]]  ] 
]
\end{forest}
\end{center}
\end{figure}
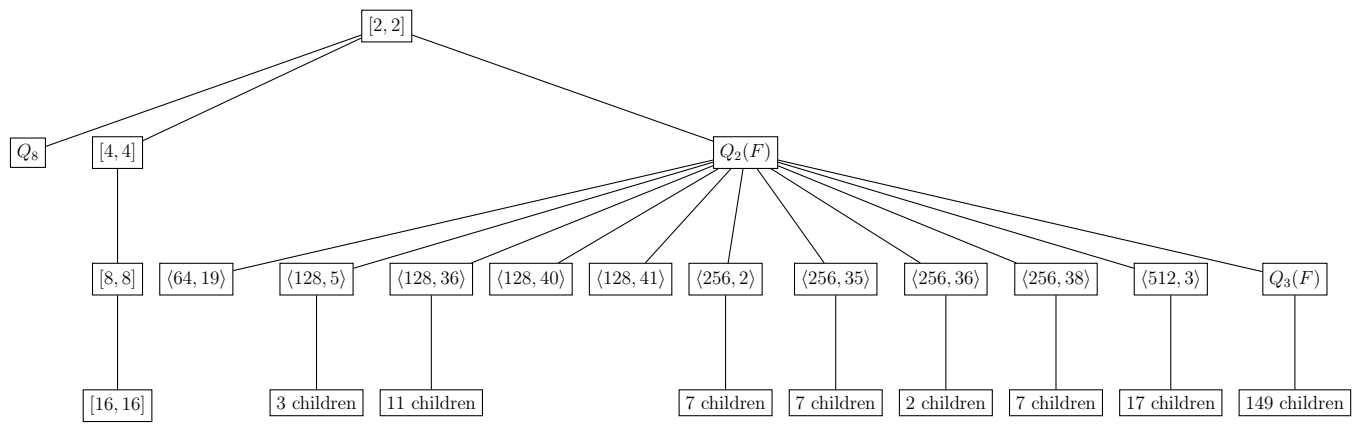


\section{IPADs}

Note that the use of IPAD in this section differs slightly from our usage in \cite{BBH} and \cite{BBH2}, where we did not include abelianizations of any index $p^2$ subgroups.

 Let $G$ be a $2$-generator pro-$2$ group. Then it has three (open) maximal subgroups $M$ and if $G$ is valid, these form an orbit under the order-three automorphism and so are isomorphic. 

\begin{definition} The IPAD of a valid $G$, $\IPAD(G)$, consists of a triple of finite abelian groups, namely $G^{ab}, M^{ab}, \Phi(G)^{ab}$.
\end{definition}

For example, for $G= \mathrm{SmallGroup}(8,4)$ ($=Q_8$), $\IPAD(G) = [[2,2]; [4]; [2]]$, whereas for $G= \mathrm{SmallGroup}(64,19)$, $\IPAD(G)=[[4,4]; [2,2,4]; [2,2,4]]$. Several small valid $2$-groups are determined up to isomorphism by their IPAD. This includes the two groups just given.

The motivation behind this definition is that we will be able to compute the IPADs of various Galois groups of interest to us, raising the question of whether these IPADs enable us to identify the Galois groups. 

We can determine which groups correspond to a given IPAD by tracking how the IPAD varies as we move down the O'Brien trees introduced in Section 2. In particular, the IPAD of a group is a quotient of the IPAD of any of its descendants and if the IPADs are the same for a group and a child, then every subsequent descendant will have the same IPAD. 

Table~\ref{IPAD_data} organizes IPADs into clumps according to the complexity of their entries and indicates which valid groups have each IPAD. Different clumps correspond to different parts of the tree: for instance the first clump consists of the quaternion group of order $8$ and the abelian groups $[2^n,2^n]$.

{\footnotesize
\begin{table}[htp]
\caption{Groups with a given IPAD.}
\label{IPAD_data}
\begin{center}
\begin{tabular}{ |c|c| } 
 \hline

$[[2,2];[2],[]$        &          $\la 4,2 \ra$ \\
\hline
$[2,2];[4];[2]$        &        $\la 8,4 \ra$ \\
\hline
$[4,4];[2,4];[2,2]$     &     $\la 16,2 \ra$ \\
\hline
$[8,8];[4,8];[4,4]$     &     $\la 64,2 \ra$ \\
\hline
$[16,16];[8,16];[8,8]$   &  $\la 256,39 \ra$ \\
\hline
    ...      & \\
    \hline
$[4,4];[2,2,4];[2,2,2]$  &   $\la 32,2 \ra$ \\
\hline
$[4,4];[2,2,4];[2,2,4]$   &  $\la 64,19 \ra$ \\
\hline
$[4,4];[2,2,4];[2,4,4]$   &  $\la 128,40 \ra, \la 128,41 \ra$ \\
\hline
$[4,4];[2,4,4];[4,4,4]$   & $\la 256,36 \ra$, its $2$ children of order $1024$ (both childless) \\
\hline
$[4,4];[2,2,8];[4,4,4]$   &  $\la 256,38 \ra$ \\
\hline
$[4,4];[2,2,8];[4,4,8]$   &  $\la 512,858 \ra$ \\
\hline
$[4,4];[2,2,8];[4,8,8]$    &  $3$ groups order $1024$ (all childless - children of $\la 256,38 \ra$) \\
\hline
$[4,4];[2,2,16];[8,8,8]$  &  $3$ groups order $2048$ (children of $\la 256,38 \ra$), all descendants of $2$ of these \\
\hline
 ... & \\
 \hline
$[8,8];[2,4,8];[2,4,4]$   &   $\la 128,5 \ra$ \\
\hline
$[8,8];[2,4,8];[2,8,8]$   &   $\la 512,104 \ra, \la 512,106 \ra$ (both childless - children of $\la 128,5 \ra$) \\
\hline
$[8,8];[2,4,8];[4,4,4]$    &  $\la 256,2 \ra, \la 512,451 \ra$ (childless - child of $\la 256,2 \ra$) \\
\hline
$[8,8];[2,4,16];[4,8,8]$ &  $2$ children of $\la 256,2 \ra$ and all their descendants \\
\hline
$[8,8];[2,8,8];[4,8,8]$  &    $2$ children of $\la 256,2 \ra$ and all their descendants \\
\hline
... & \\
\hline
$[4,4];[2,2,4];[2,2,2,2,2]$  & $\la 128,36 \ra$ and all its descendants \\
\hline
$[4,4];[2,2,8];[2,2,2,2,4]$  &  $\la 256,35 \ra$ and all its descendants (includes $\la 512,844 \ra$)\\
\hline
$[8,8];[2,4,8];[2,2,2,4,4]$  &  $\la 512,3 \ra$ and many but not all its descendants \\
\hline
... & \\

 \hline
\end{tabular}
\end{center}
\end{table}
}

\section{The class of narrow tower groups}

Our first interesting observation was that there appears to be no cyclic cubic field $K$ such that $G_2(K) \cong [4,4]$. This can be proved using embedding theory.

\begin{theorem} There is no cyclic cubic field $K$ whose $2$-class tower group is isomorphic to $[4,4]$. \end{theorem}

\begin{proof} Let $L/K$ be a Galois $2$-extension with Galois group $G$ which is unramified at all finite primes and Galois over $\mathbb{Q}$. Then there exists a Galois extension $E/\mathbb{Q}$, containing $L$, unramified over $K$ at all finite primes, with Gal$(E/L)$ isomorphic to the Schur multiplier $M(G \rtimes \Delta)$~\cite{Fr} (Corollary 2 to Theorem 3.13). In particular, for $G = [4,4]$,  $M(G \rtimes \Delta) \cong \Z/4$ and Gal$(E/\Q) \cong \la 192,4 \ra$. Complex conjugation is an element of order $1$ or $2$ inside its center (= Gal$(E/L)$) and in particular therefore does not generate it. It follows that there is a nontrivial subextension of $E/L$ unramified at all finite and infinite primes, so $L$ cannot be the maximal unramified $2$-extension of $K$.
\end{proof}

More generally, this same argument suggests that we should focus on the narrow class tower group, which we have denoted $G^+_2(K)$. 
The above embedding theory argument extends to imply that $G^+_2(K)$ belongs to a certain special family of groups. This theory is elucidated in~\cite{LWZ} where the authors call this Property E. We will call groups with Property E special.

\begin{theorem} Suppose $G$ is a finite $2$-group which is valid and satisfies $d(G)=2$. The following are equivalent.

(1) $G$ has no central extension by $\Z/2$ that is valid. 

(2) The Schur multiplier $M(G \rtimes \Delta)$ is trivial.

(3) Let $F$ be free pro-$2$ on $2$ generators $x,y$. Let $\sigma$ be the automorphism of order $3$ defined by $\sigma(x)=y^{-1}, \sigma(y)=xy^{-1}$. Let $X = \{u^{-1}\sigma(u) : u \in \Phi(F)\}$. For $u \in X$, let $\pi_u$ denote the pair $\{u,\sigma(u)\}$. $G$ is a finite quotient of $F$ by a finite number of pairs $\pi_u$.

\end{theorem}

\begin{proof}
Proposition 3.12 of \cite{LWZ}. 
\end{proof}

\begin{definition} A group satisfying any of these equivalent statements will be called \it{special}.\rm
\end{definition}

An exhaustive check finds that there are only $6$ special groups of order at most $512$, namely $\la 8,4 \ra, \la 64,19 \ra, \la 256,36 \ra, \la 512,451 \ra, \la 512,844 \ra, \la 512,858 \ra$. 

\begin{definition} Call the group \it{$k$-special}\rm \ if it can be presented as in (3) above, using no more than $k$ pairs. 
\end{definition}

Galois cohomology implies that if $G = G^+_2(K)$ with $K$ cyclic cubic, then $r(G)-d(G) \leq 2$. This follows from (11.12) in Section 11.3 of~\cite{Ko}. It follows that:

\begin{theorem} If $K$ is a cyclic cubic field with $d(Cl_2(K)) = 2$, then $G^+_2(K)$ is a $2$-special $2$-group if it is finite.
\end{theorem}

\begin{proof}
By the above, $r(G^+_2(K)) \leq 4$ and now proceed as in Proposition 3.21 of \cite{LWZ}.
\end{proof}

Of the $6$ special groups above, all but $\la 512,844 \ra$ are $2$-special ($\la 512,844 \ra$ has relation rank $6$). Working systematically, we discover that there are exactly $9$ $2$-special groups of $2$-class at most 4. With high confidence extensive searches indicate that there are $28$ $2$-special groups of $2$-class at most 5 and at least $800$ of $2$-class at most $6$. To do this we took $2$ pairs (as in Theorem 4.2 (3)) at random from the Frattini subgroup of $Q_7(F)$ $100$ million times and saved any groups this yielded that had $2$-class $6$ or less. After the first million or so, no new groups of $2$-class $5$ or less or of order $2^{20}$ or less appeared.

Table~\ref{2-special_data} indicates for each small $n$ how many $2$-special groups there are of order $2^n$ and $2$-class at most $6$. (The smallest $2$-special groups of $2$-class $7$ found so far have order $2^{17}$. With high confidence our table is complete for $n \leq 20$.) The ones for $n \leq 9$ were identified above. The two for $n=10$ are the children of $\la 256,36 \ra$ referenced in the IPAD table. 

In computing lists of $2$-special groups, we observed they are mostly childless (for instance, out of the $800$ in our list, all but $19$ are childless). More specifically, using the notion of nuclear rank defined in~\cite{O}, we conjecture the following (note that childless groups are precisely those with nuclear rank $0$):

\begin{conjecture} Every $2$-special group has nuclear rank $0$ or $2$. Its valid children are always $2$-special.
\end{conjecture}

Interestingly the special, but not $2$-special, group $\la 512,844 \ra$ has $3$ valid children, one of which is $2$-special but the other two of which are not special.

The $800$ $2$-special $2$-groups found will be ordered by the compact presentation of their standard presentation~\cite{O2}. We will use $N_k$ to refer to the $k$th group in this list. For $k = 1, 2, ..., 5$, these are respectively $\la 8,4 \ra, \la 64,19 \ra, \la 256,36 \ra, \la 512, 858 \ra, \la 512,451 \ra$. $N_6$ and $N_7$ are the two valid children of $\la 256,36 \ra$, of order $1024$. $N_8$ is the $2$-special child of $\la 512,844 \ra$, so of $2$-class $5$, of order $2048$, whereas $N_9$ and $N_{10}$ are children of $\la 256,38 \ra$ and $\la 256,2 \ra$ respectively, so of $2$-class $4$, of order $2048$. This explicitly describes all $2$-special groups of order at most $2048$. 

{\footnotesize
\begin{table}[htp]
\caption{Number of 2-special groups by order and $2$-class.}
\label{2-special_data}
\begin{center}
\begin{tabular}{ |c|c|c|c|c|c|c| } 
 \hline
 $\log_2{(\text{Order})}$ & 1 & 2 & 3 & 4 & 5 & 6 \\
 \hline
3 & 0 & 1 & 0 & 0 & 0 & 0 \\
\hline
6 & 0 & 0 & 1 & 0 & 0 & 0 \\
\hline
8 & 0 & 0 & 1 & 0 & 0 & 0 \\
\hline
9 & 0 & 0 & 0 & 2 & 0 & 0 \\
\hline
10 & 0 & 0 & 0 & 2 & 0 & 0 \\
\hline
11 & 0 & 0 & 0 & 2 & 1 & 0 \\
\hline
12 & 0 & 0 & 0 & 0 & 4 & 2 \\
\hline
13 & 0 & 0 & 0 & 0 & 4 & 0 \\
\hline
14 & 0 & 0 & 0 & 0 & 4 & 0 \\
\hline
15 & 0 & 0 & 0 & 0 & 0 & 10 \\
\hline
16 & 0 & 0 & 0 & 0 & 0 & 7 \\
\hline
17 & 0 & 0 & 0 & 0 & 6 & 28 \\
\hline
19 & 0 & 0 & 0 & 0 & 0 & 72 \\
\hline
20 & 0 & 0 & 0 & 0 & 0 & 30 \\
 \hline
\end{tabular}
\end{center}
\end{table}
}

Having seen that our narrow tower groups are $2$-special, the following conjecture suggests itself. This takes care of part (a) of the Cohen-Lenstra methodology, namely identifying the right class of groups on which to place a distribution. In the next section we do likewise for wide tower groups.

\begin{conjecture}  Every finite $2$-special $2$-group arises as the narrow tower group of some cyclic cubic field. 
\end{conjecture}

In Section 7 we investigate this conjecture by using IPADs. 

\section{The class of wide tower groups}

Recall that $G_2(K)$ is the quotient of $G^+_2(K)$ by complex conjugation, where we have to quotient by the normal subgroup generated by it and its images under the $\Delta$-action. We therefore introduce the following terminology.

\begin{definition} If $G$ is a valid $2$-group and $z \in G$ of order $1$ or $2$, then the quotient of $G$ by the normal $\Delta$-invariant subgroup generated by $z$ will be called {\em viable}.

A viable quotient of a special group will be called {\em select}. A viable quotient of a $k$-special group is called $k$-select.
\end{definition}

By Theorem 4.5, the wide tower group of a cyclic cubic field is $2$-select.

By computing all the viable quotients of the groups in our list of $2$-special groups, we find that $2$-select groups are also relatively rare. The $2$-select groups of order at most $512$ (ordered by compact presentations) are precisely the $14$ groups
$\la 4,2 \ra, \la 8,4 \ra, \la 32,2 \ra, \la 64,19 \ra, \la 128,40 \ra$, $\la 256,305 \ra, \la 256,306 \ra, \la 256,2 \ra, \la 256,38 \ra, 
\la 256,35 \ra, \la 256,36 \ra, \la 512,844 \ra, \la 512,858 \ra, \la 512,451 \ra$. 

Extending the reasoning of Theorem 4.2, there is a nice characterization of select groups along the lines of Theorem 4.2 (2).

\begin{theorem} A valid group $G$ is select if and only if the Schur multiplier $M(G \rtimes \Delta)$ is of order $1$ or $2$.
\end{theorem}

One unusual observation is that all the pairs of $2$-special and $2$-select groups obtained by taking all viable quotients of the $800$ groups in our list of $2$-special groups have elementary abelian kernel, except for one case, namely the quotient $\la 32,2 \ra$ of $N_{12}$.

We can now carry out part (a) of the Cohen-Lenstra methodology for wide tower groups.

\begin{conjecture}  Every finite $2$-select $2$-group arises as the wide tower group of some cyclic cubic field. 
\end{conjecture}

We give evidence for this conjecture in Section 7. One fascinating observation is that there do exist, however, many pairs consisting of a $2$-special group $N$ and a $2$-select group $W$ where $W$ is a viable quotient of the $N$ for which (empirically) no cyclic cubic field $K$ exists with $G^+_2(K) \cong N$ and $G_2(K) \cong W$. 

\section{Previous work on frequencies}

In this section we describe earlier work conjecturing how often the $2$-class group of a cyclic cubic field is isomorphic to a given group. There appears to be no previous work on how often the narrow $2$-class group is isomorphic to a given group.

Let $V$ denote the abstract irreducible $\F_2[\Z/3]$-module of dimension $2$ (in other words, $[2,2]$ with an automorphism of order $3$ cyclically permuting its nontrivial elements). As noted earlier, as a $\Delta$-module, $Cl_2(K)/Cl_2(K)^2$ is isomorphic to $V^d$ for some non-negative integer $d$. We set $g = 2d$. 

Malle~\cite{M08} conjectured that $d = 1$ occurs for about $14.2\%$ of cyclic cubic fields. He also conjectured that among such cases $Cl_2(K) \cong [2^n,2^n]$ with frequency $15/16^n$.
We will focus on the case $d = 1$ (and so $g = 2$) and suppose $G_2(K)^{ab} \cong [2^n,2^n]$. 

\subsection{$n=1$}

Of the $3$ valid children of $[2,2]$, only $\la 8,4 \ra$ has abelianization $[2,2]$ and it has no descendants. It follows that in the case $n=1$, $G_2(K)$ is either $[2,2]$ or $\la 8,4 \ra$. It is natural to ask how often each arises. Rubinstein-Salzedo~\cite{RS} considered this question for $K$ ramified at a single prime and found that, of $377529$ such fields, $53.7\%$ have $G_2(K) \cong [2,2]$ whereas $46.3\%$ have $G_2(K) \cong \la 8,4 \ra$. We shall return to and give some reasoning for this disparity later.

\subsection{General $n$}

In \cite{M10}, Malle conjectured that a given $2$-torsion $\mathfrak{O}$-module $H$ of even rank $2d$ should arise as $Cl_2(K)$ with frequency
$$ 2 ((2)_\infty (16)_\infty/(4)_\infty^2) 2^{d^2} (4)_{d+1}/(|H| \text{Aut}_{\mathfrak{O}}(H)) $$
where $\mathfrak{O} = \Z[\Delta]$ and
$(q)_r = \prod_{i=1}^r (1-q^{-i})$. For $d=1$, this simplifies to $1.5995/(|H| \text{Aut}_{\mathfrak{O}}(H)).$

\section{Data}

In this section, we present some data regarding the distribution of narrow and wide IPADs and discuss various conjectures. The data were obtained by starting from Malle's tables of cyclic cubic fields~\cite{Malle}. Ordering by discriminant, the first 500,000 fields $K$ with $2$-generated $2$-class group were extracted. The wide and narrow $2$-class groups of $K$ along with various extensions of degree $2$ and $4$ were computed in order to determine the wide and narrow IPADs of $K$. The distribution (both an absolute count and the relative proportions) of the narrow and wide IPAD pairs that arose can be found in Tables~\ref{table-census-abs} and~\ref{last-table}. We have divided the collection of fields into three nested sets $I_1 \subseteq I_2 \subseteq I_3$. The first contains the first 100,000 fields, the second contains the first 200,000 fields and the third contains the full set of 500,000 fields. Most of the observed proportions seem fairly stable as the set of fields is enlarged.

The computations were implemented using \pari~\cite{pari}, version 2.9.2  running on $2 \times 2.66$ GHz 6-Core Intel Xeon processors. 
The computations were run in parallel across multiple cores and took about 18,000 core hours. For some fields more than one attempt was required due to precision issues and the time spent on failed attempts is not included.

{\footnotesize
\begin{table}[htp]
\caption{Census of the most common extended IPADs.}
\label{table-census-abs}
\begin{center}

\begin{tabular}{|l|l |c|c|c|}
\hline
Narrow & Wide  & $I_1$ & $I_2$ & $I_3$\\
\hline
$[2,2]$; $[4]$; $[2]$  &  $[2,2]$; $[2]$; $[]$  & 50500 & 101172 & 251883 \\
\hline
$[2,2]$; $[4]$; $[2]$  &  $[2,2]$; $[4]$; $[2]$  & 25022 & 49829 & 124661 \\
\hline
$[4,4]$; $[2,2,4]$; $[2,2,4]$  &  $[2,2]$; $[4]$; $[2]$  & 18006 & 36225 & 91413
\\
\hline
$[4,4]$; $[2,4,4]$; $[4,4,4]$  &  $[4,4]$; $[2,2,4]$; $[2,2,2]$  & 1225 & 2411 &
6098 \\
\hline
$[4,4]$; $[2,2,8]$; $[2,2,2,2,4]$  &  $[4,4]$; $[2,2,4]$; $[2,2,2]$  & 1220 & 
2357 & 5894 \\
\hline
$[4,4]$; $[2,2,8]$; $[2,2,2,2,4]$  &  $[4,4]$; $[2,2,4]$; $[2,2,4]$  & 586 & 
1169 & 3018 \\
\hline
$[8,8]$; $[2,4,8]$; $[4,4,4]$  &  $[4,4]$; $[2,2,4]$; $[2,2,4]$  & 578 & 1180 & 
2982 \\
\hline
$[4,4]$; $[2,2,8]$; $[4,4,8]$  &  $[4,4]$; $[2,2,4]$; $[2,2,4]$  & 625 & 1212 & 
2967 \\
\hline
$[4,4]$; $[2,4,4]$; $[4,4,4]$  &  $[4,4]$; $[2,2,4]$; $[2,4,4]$  & 415 & 826 & 
2061 \\
\hline
$[4,4]$; $[2,2,8]$; $[2,2,2,2,4]$  &  $[4,4]$; $[2,2,4]$; $[2,2,2,2,2]$  & 397 &
832 & 2050 \\
\hline
$[8,8]$; $[4,4,8]$; $[2,2,4,4,8]$  &  $[4,4]$; $[2,4,4]$; $[4,4,4]$  & 196 & 395
& 1044 \\
\hline
$[4,4]$; $[2,2,8]$; $[2,2,2,2,4]$  &  $[4,4]$; $[2,2,8]$; $[2,2,2,2,4]$  & 187 &
416 & 983 \\
\hline
$[8,8]$; $[4,4,8]$; $[2,2,4,4,8]$  &  $[8,8]$; $[2,4,8]$; $[4,4,4]$  & 140 & 234
& 590 \\
\hline
$[8,8]$; $[4,4,8]$; $[2,2,4,4,8]$  &  $[4,4]$; $[2,2,8]$; $[2,2,2,2,4]$  & 116 &
232 & 570 \\
\hline
$[8,8]$; $[2,4,16]$; $[4,8,8]$  &  $[8,8]$; $[2,4,8]$; $[4,4,4]$  & 79 & 157 & 
379 \\
\hline
$[4,4]$; $[2,2,16]$; $[8,8,8]$  &  $[4,4]$; $[2,2,8]$; $[4,4,4]$  & 76 & 149 & 
372 \\
\hline
$[8,8]$; $[4,8,8]$; $[2,2,8,8,8]$  &  $[4,4]$; $[2,4,4]$; $[4,4,4]$  & 74 & 145 
& 352 \\
\hline
$[8,8]$; $[4,4,8]$; $[2,2,4,4,8]$  &  $[4,4]$; $[2,2,8]$; $[4,4,4]$  & 63 & 117 
& 340 \\
\hline
$[8,8]$; $[2,8,8]$; $[4,8,8]$  &  $[8,8]$; $[2,4,8]$; $[4,4,4]$  & 42 & 87 & 213
\\
\hline
$[8,8]$; $[4,4,8]$; $[2,2,4,4,8]$  &  $[4,4]$; $[2,2,8]$; $[4,4,8]$  & 41 & 81 &
212 \\
\hline
Other IPADs (73 types) &   & 412 & 774 & 1918 \\
\hline
\hline
Total &   & 100000 & 200000 & 500000 \\
\hline
\end{tabular}
\end{center}
\end{table}
}

{\footnotesize
\begin{table}[htp]
\caption{Relative proportions of the most common extended IPADs.}
\label{last-table}
\begin{center}

\begin{tabular}{|l|l |c|c|c|}
\hline
Narrow & Wide  & $I_1$ & $I_2$ & $I_3$\\
\hline
$[2,2]$; $[4]$; $[2]$  &  $[2,2]$; $[2]$; $[]$  & 0.5050 & 0.5059 & 0.5038 \\
\hline
$[2,2]$; $[4]$; $[2]$  &  $[2,2]$; $[4]$; $[2]$  & 0.2502 & 0.2491 & 0.2493 \\
\hline
$[4,4]$; $[2,2,4]$; $[2,2,4]$  &  $[2,2]$; $[4]$; $[2]$  & 0.1801 & 0.1811 & 
0.1828 \\
\hline
$[4,4]$; $[2,4,4]$; $[4,4,4]$  &  $[4,4]$; $[2,2,4]$; $[2,2,2]$  & 0.0123 & 
0.0121 & 0.0122 \\
\hline
$[4,4]$; $[2,2,8]$; $[2,2,2,2,4]$  &  $[4,4]$; $[2,2,4]$; $[2,2,2]$  & 0.0122 & 
0.0118 & 0.0118 \\
\hline
$[4,4]$; $[2,2,8]$; $[2,2,2,2,4]$  &  $[4,4]$; $[2,2,4]$; $[2,2,4]$  & 0.0059 & 
0.0058 & 0.0060 \\
\hline
$[8,8]$; $[2,4,8]$; $[4,4,4]$  &  $[4,4]$; $[2,2,4]$; $[2,2,4]$  & 0.0058 & 
0.0059 & 0.0060 \\
\hline
$[4,4]$; $[2,2,8]$; $[4,4,8]$  &  $[4,4]$; $[2,2,4]$; $[2,2,4]$  & 0.0063 & 
0.0061 & 0.0059 \\
\hline
$[4,4]$; $[2,4,4]$; $[4,4,4]$  &  $[4,4]$; $[2,2,4]$; $[2,4,4]$  & 0.0042 & 
0.0041 & 0.0041 \\
\hline
$[4,4]$; $[2,2,8]$; $[2,2,2,2,4]$  &  $[4,4]$; $[2,2,4]$; $[2,2,2,2,2]$  & 
0.0040 & 0.0042 & 0.0041 \\
\hline
$[8,8]$; $[4,4,8]$; $[2,2,4,4,8]$  &  $[4,4]$; $[2,4,4]$; $[4,4,4]$  & 0.0020 & 
0.0020 & 0.0021 \\
\hline
$[4,4]$; $[2,2,8]$; $[2,2,2,2,4]$  &  $[4,4]$; $[2,2,8]$; $[2,2,2,2,4]$  & 
0.0019 & 0.0021 & 0.0020 \\
\hline
$[8,8]$; $[4,4,8]$; $[2,2,4,4,8]$  &  $[8,8]$; $[2,4,8]$; $[4,4,4]$  & 0.0014 & 
0.0012 & 0.0012 \\
\hline
$[8,8]$; $[4,4,8]$; $[2,2,4,4,8]$  &  $[4,4]$; $[2,2,8]$; $[2,2,2,2,4]$  & 
0.0012 & 0.0012 & 0.0011 \\
\hline
$[8,8]$; $[2,4,16]$; $[4,8,8]$  &  $[8,8]$; $[2,4,8]$; $[4,4,4]$  & 0.0008 & 
0.0008 & 0.0008 \\
\hline
$[4,4]$; $[2,2,16]$; $[8,8,8]$  &  $[4,4]$; $[2,2,8]$; $[4,4,4]$  & 0.0008 & 
0.0007 & 0.0007 \\
\hline
$[8,8]$; $[4,8,8]$; $[2,2,8,8,8]$  &  $[4,4]$; $[2,4,4]$; $[4,4,4]$  & 0.0007 & 
0.0007 & 0.0007 \\
\hline
$[8,8]$; $[4,4,8]$; $[2,2,4,4,8]$  &  $[4,4]$; $[2,2,8]$; $[4,4,4]$  & 0.0006 & 
0.0006 & 0.0007 \\
\hline
$[8,8]$; $[2,8,8]$; $[4,8,8]$  &  $[8,8]$; $[2,4,8]$; $[4,4,4]$  & 0.0004 & 
0.0004 & 0.0004 \\
\hline
$[8,8]$; $[4,4,8]$; $[2,2,4,4,8]$  &  $[4,4]$; $[2,2,8]$; $[4,4,8]$  & 0.0004 & 
0.0004 & 0.0004 \\
\hline
Other IPADs (73 types) &  & 0.0041 & 0.0039 & 0.0038  \\
\hline
\end{tabular}
\end{center}
\end{table}
}

In Table~\ref{explanation}, for the most common pairs of narrow/wide IPADs, with the help of Table 1, we infer what the corresponding narrow tower group and wide tower group must be. 
The last column gives a guess (based on the observed frequency) for the proportion of cyclic cubic fields that have those narrow and wide IPADs.

{\footnotesize
\begin{table}[htp]
\caption{Explanation and conjectural frequency of IPAD pairs.}\label{explanation}
\begin{center}
\begin{tabular}{ |c|c|c|c|c| } 
 \hline
\text{Narrow} & \text{Wide} & \text{Explanation} & \text{Obs freq} & \text{Conj freq}  \\
 \hline
$[2,2];[4];[2]$ & $[2,2];[2];[]$ & $N_1 \rightarrow \la 4,2 \ra$ & 0.5038 & 1/2 \\
\hline
$[2,2];[4];[2]$ & $[2,2];[4];[2]$ & $N_1 \rightarrow \la 8,4 \ra$ & 0.2493 & 1/4 \\
\hline
$[4,4];[2,2,4];[2,2,4]$ & $[2,2];[4];[2]$ & $N_2 \rightarrow \la 8,4 \ra$ & 0.1828 & 3/16 \\
\hline
$[4,4];[2,4,4];[4,4,4]$ & $[4,4];[2,2,4];[2,2,2]$ & $N_3 \rightarrow \la 32,2 \ra$ & 0.0122 & 3/256 \\
\hline
$[4,4];[2,2,8];[2,2,2,2,4]$ & $[4,4];[2,2,4];[2,2,2]$ & 7.2 below & 0.0118 & 3/256 \\
\hline
$[4,4];[2,2,8];[2,2,2,2,4]$ & $[4,4];[2,2,4];[2,2,4]$ & 7.3 below & 0.0060 & 3/512 \\
\hline
$[8,8];[2,4,8];[4,4,4]$ & $[4,4];[2,2,4];[2,2,4]$ & $N_5 \rightarrow \la 64,19 \ra$ & 0.0060 & 3/512 \\
\hline
$[4,4];[2,2,8];[4,4,8]$ & $[4,4];[2,2,4];[2,2,4]$ & $N_4 \rightarrow \la 64,19 \ra$ & 0.0059 & 3/512 \\
\hline
$[4,4];[2,4,4];[4,4,4]$ & $[4,4];[2,2,4];[2,4,4]$ & $N_3   \ \text{or} \ N_6 \ \text{or} \ N_7 \rightarrow \la 128,40 \ra$ & 0.0041 & 1/256 \\
 \hline
\end{tabular}
\end{center}
\end{table}
}

Note that this proves Conjecture 4.7 for $N_k$ for $k=1,2,3,4,5$ and Conjecture 5.3 for the first five $2$-select groups. The $16$th, $19$th, and $20$th rows of Table 3 show that $\la 256,38 \ra, \la 512,451 \ra,$ and $\la 512,858 \ra$ respectively satisfy Conjecture 5.3. 

We can also ask which viable quotients of $N_k$ arise as wide tower groups. For $N_1$ both do, but for $N_2$, whereas both $\la 64,19 \ra$ and $\la 32,2 \ra$ are viable quotients, their IPADs have not been seen for any cyclic cubic field as the wide IPAD paired with the narrow IPAD of $N_2$. We therefore believe that if the narrow group is $N_2$ (i.e, $\la 64,19 \ra$), then the wide group must be $\la 8,4 \ra$. With the same reasoning we make the following conjecture. The only slightly different case is (2), since $N_3$ has both $\la 128,40 \ra$ and $\la 32,2 \ra$ as viable quotients with fields existing with the corresponding pair of narrow/wide IPADs. If, however, we check the first $6$ such fields $K$ with wide tower group $\la 128,40 \ra$, by looking at the narrow 2-class group of a certain octic extension of $K$, we see that the narrow tower group is not $N_3$ (so is $N_6$ or $N_7$).

\begin{conjecture} 

(1) If a cyclic cubic field has narrow tower group $N_2 = \la 64,19 \ra$, then its wide tower group is $\la 8,4 \ra$.

(2) If a cyclic cubic field has narrow tower group $N_3 = \la 256,36 \ra$, then its wide tower group is $\la 32,2 \ra$.

(3) If a cyclic cubic field has narrow tower group $N_4$, then its wide tower group is $\la 64,19 \ra$.

(4) If a cyclic cubic field has narrow tower group $N_5$, then its wide tower group is $\la 64,19 \ra$.

(5) If a cyclic cubic field has narrow tower group $N_6$, then its wide tower group is $\la 128,40 \ra$.

(6) If a cyclic cubic field has narrow tower group $N_7$, then its wide tower group is $\la 128,40 \ra$.

\end{conjecture}

\begin{remark} Consider the $5$th row in Table~\ref{explanation}. The wide tower group is $\la 32,2 \ra$. Of the $800$ $2$-special groups, only $N_8$ and $N_{12}$ have this narrow IPAD and a viable quotient isomorphic to $\la 32,2 \ra$. We can distinguish between these two cases by computing narrow class groups of certain octic extensions of $K$. We did this for $6$ fields with this pair of IPADs and in each case the narrow tower group turned out not to be $N_8$. We also found some larger $2$-special groups, of $2$-class 7 and order $2^{20}$, with this IPAD and a viable quotient isomorphic to $\la 32,2 \ra$.
\end{remark}

\begin{remark} Consider the $6$th row in Table~\ref{explanation}. The wide tower group is $\la 64,19 \ra$. Of the $800$ $2$-special groups, only $N_8$ has this narrow IPAD and a viable quotient isomorphic to $\la 64,19 \ra$. We do not know whether there exist larger $2$-special groups with this property.
\end{remark}

\begin{remark} Let $G$ be a $2$-select group. It is fundamental to ask which $2$-special groups have $G$ as a viable quotient. The above remarks consider this briefly for $G = \la 32,2 \ra$ and $\la 64,19 \ra$ respectively. The last section of this paper indicates that if $G$ is a certain group of order $2048$, then it is a viable quotient of some infinite $2$-special group. Extensive investigations suggest that $\la 4,2 \ra$ is a viable quotient of only $N_1$ and that $\la 8,4 \ra$ is a viable quotient of only $N_1$ and $N_2$. We do not know whether every $2$-select group is a viable quotient of only finitely many $2$-special groups. This will be key in understanding the joint distribution we seek.
\end{remark}

In some cases it is very clear how often a certain $2$-special $2$-group arises as $G^+_2(K)$ as $K$ varies. For instance, for $\la 8,4 \ra$, the frequency is almost exactly $3/4$, whereas for $\la 64,19 \ra$ the frequency is almost exactly $3/16$. What is perhaps surprising is that the frequencies do not match the proportion of pairs of pairs yielding a given $2$-special group by the method of Theorem 4.2 (3).

Ideally, we would like a formula for the proportion of fields with a given narrow $2$-class tower group $N$ and given wide $2$-class tower group $W$, in terms of $N$ and $W$, but since we have yet to understand which pairs $(N,W)$ arise, even a conjectural formula is currently out of reach. We can, however, use the data to make some related conjectures regarding proportions of fields.

\subsection{Proportion for narrow $2$-class $3$ quotients}

More detailed analysis allows us to make reliable conjectures as to the frequencies with which $Q_3(G^+_2(K))$ arises as $K$ varies. 

\begin{lemma} Among $2$-special groups $G$, the form of $\IPAD(G)$ determines $Q_3(G)$. In particular, with the help of Table~\ref{IPAD_data},

(1) if $\IPAD(G) = [[2,2];[4];[2]]$, then $Q_3(G) = \la 8,4 \ra$;

(2) if $\IPAD(G) = [[4,4];[2,2,4];[2,2,4]]$, then $Q_3(G) = \la 64,19 \ra$;

(3) if $\IPAD(G) = [[\ast,\ast];[-,-,-];[-,-,-]]$, then $Q_3(G) = \la 256,2 \ra$;

(4) if $\IPAD(G) = [[4,4];[2,2,8];[2,2,2,2,4]]$, then $Q_3(G) = \la 256,35 \ra$;

(5)  if $\IPAD(G) = [[4,4];[2,4,4];[4,4,4]]$,  then $Q_3(G) = \la 256,36 \ra$;

(6) if $\IPAD(G) = [[4,4];[-,-,-];[-,-,-]]$, other than above, then $Q_3(G) = \la 256,38 \ra$;

(7) if $\IPAD(G) = [[\ast,\ast];[-,-,-];[-,-,-,-,-]]$, then $Q_3(G) = Q_3(F)$ of order $1024$.

Here $\ast$ stands for an integer $\geq 8$ and $-$ stands for an integer $\geq 2$.

\end{lemma}

Using this and our data on occurrence of IPADs, we obtain the following conjectural frequencies for 
$Q_3(G^+_2(K))$ given in Table~\ref{narrow_conj}.

{\footnotesize
\begin{table}[htp]
\caption{Conjectural frequency for $Q_3(G^+_2(K))$.}
\label{narrow_conj}
\begin{center}
\begin{tabular}{ |c|c|c| } 
 \hline
\text{Group} & \text{Obs freq} & \text{Conj freq}\\
\hline
$ \la 8,4 \ra$ & 0.7531 & 3/4 \\
\hline
$ \la 64,19 \ra$ & 0.1828 & 3/16 \\
\hline
$ \la 256,2 \ra$ & 0.0080 & 1/128 \\
\hline
$ \la 256,35 \ra$ & 0.0239 & 3/128 \\
\hline
$ \la 256,36 \ra$ & 0.0163 & 1/64 \\
\hline
$ \la 256,38 \ra$ & 0.0079 & 1/128 \\
\hline
$ Q_3(F)$ & 0.0079 & 1/128 \\
 \hline
\end{tabular}
\end{center}
\end{table}
}

It is a mystery as to why such similar groups as $\la 256,35 \ra$ and $\la 256,38 \ra$ do not arise equally often.

\subsection{Proportion for wide $2$-class $3$ quotients}

Likewise, we can make reliable conjectures as to the frequencies with which $Q_3(G_2(K))$ arises as $K$ varies. 

\begin{lemma} Among $2$-select groups, the form of $\IPAD(G)$ determines $Q_3(G)$. In particular, with the help of Table~\ref{IPAD_data},

(1) if $\IPAD(G) = [[2,2];[2];[]]$, then $Q_3(G) = \la 4,2 \ra$;

(2) if $\IPAD(G) = [[2,2];[4];[2]]$, then $Q_3(G) = \la 8,4 \ra$;

(3) if $\IPAD(G) = [[4,4];[2,2,4];[2,2,2]]$, then $Q_3(G) = \la 32,2 \ra$;

(4) if $\IPAD(G) = [[4,4];[2,2,4];[2,2,4]]$, then $Q_3(G) = \la 64,19 \ra$;

(5) if $\IPAD(G) = [[4,4];[2,2,4];[2,2,2,2,2]]$, then $Q_3(G) = \la 128,36 \ra$;

(6) if $\IPAD(G) = [[4,4];[2,2,4];[2,4,4]]$, then $Q_3(G) = \la 128,40 \ra$;

(7) if $\IPAD(G) = [[\ast,\ast];[-,-,-];[-,-,-]]$, then $Q_3(G) = \la 256,2 \ra$;

(8) if $\IPAD(G) = [[4,4];[2,2,8];[2,2,2,2,4]]$, then $Q_3(G) = \la 256,35 \ra$;

(9)  if $\IPAD(G) = [[4,4];[2,4,4];[4,4,4]]$,  then $Q_3(G) = \la 256,36 \ra$;

(10) if $\IPAD(G) = [[4,4];[-,-,-];[-,-,-]]$, other than above, then $Q_3(G) = \la 256,38 \ra$;

(11) if $\IPAD(G) = [[\ast,\ast];[-,-,-];[-,-,-,-,-]]$, then $Q_3(G) = Q_3(F)$ of order $1024$.

Here $\ast$ stands for an integer $\geq 8$ and $-$ stands for an integer $\geq 2$.

\end{lemma}

{\footnotesize
\begin{table}[htp]
\caption{Conjectural frequency for $Q_3(G_2(K))$.}
\label{wide_conj}
\begin{center}
\begin{tabular}{ |c|c|c| } 
 \hline
\text{Group} & \text{Obs freq} & \text{Conj freq}\\
\hline
$ \la 4,2 \ra$ & 0.5038 & 1/2 \\
\hline
$ \la 8,4 \ra$ & 0.4321 & 7/16 \\
\hline
$ \la 32,2 \ra$ & 0.0240 & 3/128 \\
\hline
$ \la 64,19 \ra$ & 0.0179 & 9/512 \\
\hline
$\la 128,36 \ra$ & 0.0041 & 1/256 \\
\hline
$ \la 128,40 \ra$ & 0.0041 & 1/256 \\
\hline
$ \la 256,2 \ra$ & 0.0036 & 7/2048 \\
\hline
$ \la 256,35 \ra$ & 0.0034 & 7/2048 \\
\hline
$ \la 256,36 \ra$ & 0.0028 & 3/1024 \\
\hline
$ \la 256,38 \ra$ & 0.0035 & 7/2048 \\
\hline
$ Q_3(F)$ & 0.0006 & 1/2048 \\
 \hline
\end{tabular}
\end{center}
\end{table}
}

\subsection{Joint distribution for narrow/wide $2$-class $3$ quotients}

Since we actually compute how often various narrow and wide IPAD pairs occur together, we can refine the computations of the last two subsections to give a conjectural joint distribution on the $77$ possibilities given by Lemmas 7.4 and 7.5. Summing its columns and rows gives the marginal distributions given in Tables~\ref{narrow_conj} and~\ref{wide_conj}. A key step in guessing the joint distribution for the narrow and wide tower groups would be to understand Table~\ref{joint_conj}.

{\footnotesize
\begin{table}[htp]
\caption{Conjectural joint distribution for $(Q_3(G_2(K)),Q_3(G^+_2(K)))$.}
\label{joint_conj}
\begin{center}
\begin{tabular}{ |c|c|c|c|c|c|c|c| } 
 \hline
\text{Narrow} & $\la 8,4 \ra$ & $\la 64,19 \ra$ & $\la 256,2 \ra$ & $\la 256,35 \ra$ & $\la 256,36 \ra$ & $\la 256,38 \ra$ & $Q_3(F)$ \\
\hline
\text{Wide} & & & & & & & \\
\hline
$\la 4,2 \ra$ & 1/2 & 0 & 0 & 0 & 0 & 0 & 0 \\ 
 \hline
 $\la 8,4 \ra$ & 1/4 & 3/16 & 0 & 0 & 0 & 0 & 0 \\
 \hline
 $\la 32,2 \ra$ & 0 & 0 & 0 & 3/256 & 3/256 & 0 & 0 \\
 \hline
 $\la 64,19 \ra$ & 0 & 0 & 3/512 & 3/512 & 0 & 3/512 & 0 \\
 \hline
 $\la 128,36 \ra$ & 0 & 0 & 0 & 1/256 & 0 & 0 & 0 \\
 \hline
 $\la 128,40 \ra$ & 0 & 0 & 0 & 0 & 1/256 & 0 & 0 \\
 \hline
 $\la 256,2 \ra$ & 0 & 0 & 1/512 & 0 & 0 & 0 & 3/2048 \\
 \hline
 $\la 256,35 \ra$ & 0 & 0 & 0 & 1/512 & 0 & 0 & 3/2048 \\
 \hline
 $\la 256,36 \ra$ & 0 & 0 & 0 & 0 & 0 & 0 & 3/1024 \\
 \hline
 $\la 256,38 \ra$ & 0 & 0 & 0 & 0 & 0 & 1/512 & 3/2048 \\
 \hline
 $Q_3(F)$ & 0 & 0 & 0 & 0 & 0 & 0 & 1/2048 \\
 \hline
\end{tabular}
\end{center}
\end{table}
}

\subsection{Narrow/wide $2$-class group pairs}

As noted in Lemma 2.1, the number of generators of the wide and narrow $2$-class groups are the same. It is, however, possible for the $2$-class groups not to be isomorphic, although the difference cannot be large. In fact the ratio of their orders is always $1$ or $4$.

This follows from the exact sequence
$$ 1 \rightarrow (\pm 1)^3/\text{sgn}_\infty(O_K^*) \rightarrow Cl_2^+(K) \rightarrow Cl_2(K) \rightarrow 1 $$
\cite{BVV} and the fact that the orders of $Cl_2^+(K)$ and $Cl_2(K)$ are both squares. 

{\footnotesize
\begin{table}[htp]
\caption{Conjectural frequency for narrow/wide $2$-class group pairs.}\label{abelian_conj}
\begin{center}
\begin{tabular}{ |c|c|c|c| } 
 \hline
\text{Narrow} & \text{Wide} & \text{Obs freq} & \text{Conj freq}  \\
 \hline
 $[2,2]$ & $[2,2]$ & 0.7531 & 3/4 \\
 \hline
 $[4,4]$ & $[2,2]$ & 0.1828 & 3/16 \\
 \hline
 $[4,4]$ & $[4,4]$ & 0.0481 & 3/64 \\
 \hline
 $[8,8]$ & $[4,4]$ & 0.0118 & 3/256 \\
 \hline
 $[8,8]$ & $[8,8]$ & 0.0032 & 3/1024 \\
 \hline
 $[16,16]$ & $[8,8]$ & 0.0007 & 3/4096 \\
 \hline
 $[16,16]$ & $[16,16]$ & 0.0002 & 3/16384 \\
 \hline
 ... & ... & & \\
\hline
\end{tabular}
\end{center}
\end{table}
}

\begin{conjecture} Among cyclic cubic fields $K$ with $2$-generator $2$-class group, following Table~\ref{abelian_conj}, we conjecture that $Cl^+_2(K) \cong [2^n,2^n] \cong Cl_2(K)$ for a proportion $3/2^{4n-2}$ of fields and that $Cl^+_2(K) \cong [2^{n+1},2^{n+1}]$ and $Cl_2(K) \cong [2^n,2^n]$ for a proportion $3/2^{4n}$ of fields. 
\end{conjecture}

Note that this is consistent with Malle's heuristics (Section 6) since $3/2^{4n-2} + 3/2^{4n} = 15/16^n$.
It is also consistent with Conjecture 6.2.3 of~\cite{BVV} since each predicts that the narrow and wide $2$-class groups should agree $4/5$ of the time.

Looking at these data closer, when $Cl^+_2(K) \cong [2,2]$ and $Cl_2(K) \cong [2,2]$, we know that $G^+_2(K) \cong \la 8,4 \ra$ and $G_2(K) \cong \la 8,4 \ra$ or $[2,2]$. From the first two lines of Table~\ref{last-table}, the first possibility empirically occurs half as often as the second (so of the $3/4$, the first accounts for $1/2$ and the second for $1/4$). When $Cl^+_2(K) \cong [4,4]$ and $Cl_2(K) \cong [2,2]$, the third line of Table~\ref{explanation} actually indicates that $G^+_2(K) \cong \la 64,19 \ra$ and $G_2(K) \cong \la 8,4 \ra$. 

Combining these, then, as $K$ varies, $G_2(K) \cong [2,2]$ for $1/2$ of the time whereas $G_2(K) \cong \la 8,4 \ra$ for $1/4 + 3/16 = 7/16$ of the time. This matches Rubinstein-Salzedo's~\cite{RS} observations well, since $8/15$ is close to $53.7\%$. We do, however, still make the following conjecture.

\begin{conjecture}
Precisely one half of unramified $[2,2]$-extensions of cyclic cubic fields extend to unramified $\la 8,4 \ra$-extensions.
\end{conjecture}

How do we reconcile this with the Rubinstein-Salzedo observation? The point is that the other cases with $g=2$ (where $G_2(K)^{ab}$ is $[4,4]$ or $[8,8]$ or ...) necessarily have $G_2(K)$ being a descendant of $Q_2(F)$ (see the earlier discussion of the $g=2$ O'Brien tree, as in Figure 1). It therefore has $\la 8,4 \ra$ as a quotient. It follows that all these other cases, accounting for $1/16$ of all cases, give rise to an unramified $\la 8,4 \ra$-extension, and so there are overall $7/16 + 1/16 = 1/2$ such cases!

\section{Speculation on infinite narrow class towers}

There is no explicit example known of a Galois group of an infinite $p$-extension of a number field, ramified at only finitely many primes, not including any primes above $p$. There has, however, been plenty of speculation as to what kind of pro-$p$ group this Galois group can or cannot be~\cite{B}. The goal of this section is to see whether our theory sheds any light on this matter. 

Suppose that a cyclic cubic field $K$ has an infinite narrow $2$-class tower. Since every $g$-generator $(g+2)$-relator pro-$p$ group is infinite if $g \geq 6$ by Golod-Shafarevich, our method below provides little useful information for large $g$. The simplest situation to consider is where $G^+_2(K)$ is a $2$-generator, so $4$-relator, pro-$2$ group. Note that we do not know of any such $K$ with infinite narrow $2$-class tower, but on the other hand we do not know one does not exist. 

What can we say then about $G^+_2(K)$, if such a $K$ does exist? By class field theory every open subgroup has finite abelianization (sometimes called the FAb property). Indeed, the Fontaine-Mazur conjecture~\cite{FM} applies to say that $G^+_2(K)$ should have no $2$-adic representations with infinite image.

Also, we might suspect that $G^+_2(K)$ is a $2$-special pro-$2$ group, using the obvious extension of Definition 4.4  to infinite groups. This allows us to carry out the experiment whereby we choose $u_1,u_2$ from $\Phi(F)$ and consider the $2$-special group $G:=\la x,y \mid v_1,\sigma(v_1),v_2,\sigma(v_2) \ra$, where $v_i = u_i^{-1}\sigma(u_i)$. Choosing random words (of some bounded length) in $\Phi(F)$ a million times, we put the corresponding $G$ through two filters.

The first filter checks that the order of $Q_c(G)$ continues to grow as $c$ increases. We check this at least up to $c=32$. The second filter checks that subgroups of $G$ of small index have finite abelianization. In practice, Magma allows us to check this for subgroups of index $1,2,4,8,16$ (with cores of $2$-power index so they do arise in the pro-$2$ group given by the presentation). 

While these are not complete checks that $G$ is infinite and FAb, they do filter out most presentations. We find that the following two types of groups are left. We study them via the sequence $\log_2 |Q_c(G)|$. The groups $G$ all have IPAD $[[4,4];[2,2,8];[2,2,2,2,4]]$.

A typical sequence for the first type of group is $2, 5, 8, 11, 15, 18, 23, 28, 30, 33, 36, 38, 41, 44, $ $46, 49, 52, 54, 57, 60, 62, 65, 68, 70, 73, 76, 78, 81, 84, 86, 89, 92, ...$. It is characterized by ultimate periodicity in the differences, which is a strong indication that $G$ is indeed infinite. It is, however, also an indication that the group may be $2$-adic analytic, which would be forbidden by Fontaine-Mazur. One can also compute abelianizations of normal subgroups of index at least up to $4096$ and see that their ranks are apparently bounded, indicating finite Pr\"ufer rank. The pro-$2$ groups of finite Pr\"ufer rank are exactly the $2$-adic analytic ones.

Apparently, then, the only surviving type of group is the second, whose sequence is $2, 5, 8, 11, 14, 16, 20, 24, 28, 31, 35, 39, 45, 53, 60, 65, 71, 78, 
88, 99, 109, 119, 131, 146, 162, 178, 192, \ \ \ $ 
$ 206, 224, 247, 269, 293, ...$ or similar. We do not have an exact description of this sequence. One example is for $u_1=(y^{-1}, x^{-1})^2, u_2=(yx^{-1}y^{-1}x^{-1})^2$. An element of order $2$ here is, for example, $(x^2y^{-2})^2$. The corresponding viable quotients have IPAD $[[4,4];[2,2,8];[2,2,2,2,4]]$, but what is most interesting about them is that they are finite, in fact isomorphic to a certain $2$-select group of $2$-class 5 and order $2^{11}.$ Whatever element of order $2$ is used appears to lead to the same group. \\[-2 pt]

\noindent{\sc Acknowledgements.}  We thank Brandon Alberts, Jordan Ellenberg, Farshid Hajir, Yuan Liu, John Voight, Jiuya Wang, and Melanie Matchett Wood for helpful conversations and feedback.  
The work of MRB was partially supported by summer Lenfest Grants from Washington and Lee University.


\begin{thebibliography}{99}

\bibitem{AF} J. V. Armitage, A. Fr\"ohlich, {\em Class numbers and unit signatures}, Mathematika {\bf 14} (1967), 94--98.
\bibitem{magma} W. Bosma, J. J. Cannon, C. Playoust, {\em The Magma algebra system. I. The user language}, J. Symbolic Comput., {\bf 24} (1997), 235--265.
\bibitem{B} N.~Boston, {\em Galois groups of tamely ramified $p$-extensions}, J. Th\'eorie des Nombres de Bordeaux, {\bf 19} (2007), no.~1, 59--70.
\bibitem{BBH} N.~Boston, M.R.~Bush, F.~Hajir, \emph{Heuristics for $p$-class towers of imaginary quadratic fields}, Math. Ann. {\bf 368} (2017), no. 1--2, 633--669.
\bibitem{BBH2} N.~Boston, M.R.~Bush, F.~Hajir, \emph{Heuristics for $p$-class towers of real quadratic fields}, Jour. Inst. Math. Jussieu, (2019), 1--24.
\bibitem{BVV} B.~Breen, I.~Varma, J.~Voight, {\em On unit signatures and narrow class groups of odd abelian number fields: structure and heuristics}, https://arxiv.org/abs/1910.00449, preprint, 2019.
\bibitem{CL2} H.~Cohen, H. W.~Lenstra, Jr., {\em Heuristics on class
  groups of number fields}, pp. 33--62 in: Number theory,
  Noordwijkerhout 1983, LNM {\bf 1068}, Springer, Berlin, 1984.
 \bibitem{FM} J.-M.~Fontaine, B.~Mazur, {\em Geometric Galois representations}, in Coates, John; Yau., S.-T. (eds.), Elliptic curves, modular forms, \& Fermat's last theorem (Hong Kong, 1993), Series in Number Theory, 1, Int. Press, Cambridge, MA, pp. 41--78.
 \bibitem{Fr} A.~Fr\"ohlich, {\em Central Extensions, Galois Groups, and Ideal Class Groups of Number Fields}, Contemp. Math., 24, (1983), Amer. Math. Soc, Providence. 
 \bibitem{Ga} D.~Garton, {\em Random matrices, the Cohen-Lenstra heuristics, and roots of unity}, Algebra \& Number Theory, {\bf 9} (2015), 149--171.
 \bibitem{Kl} J.~Klys, {\em The distribution of $p$-torsion in degree $p$ cyclic fields}, Algebra \& Number Theory, {\bf 14} (2020), 815--854.
\bibitem{Ko} H.~Koch, ``Galois theory of $p$-extensions.'' With a foreword by I. R. Shafarevich. Translated from the 1970 German original by Franz Lemmermeyer. With a postscript by the author and Lemmermeyer. Springer Monographs in Mathematics. Springer-Verlag, Berlin, 2002.
\bibitem{LWZ} Y. ~Liu, M. ~Matchett Wood, D.~Zureick-Brown, {\em A predicted distribution for Galois groups of maximal unramified extensions}, https://arxiv.org/abs/1907.05002, preprint, 2019.
\bibitem{M08} G.~Malle, {\em Cohen-Lenstra heuristic and roots of unity},  J. Number Theory {\bf 128} (2008), 2823--2835.
\bibitem{M10} G.~Malle, {\em On the distribution of class groups of number fields}, Experimental Math {\bf 19(4)} (2010), 465--474.
\bibitem{Malle} G.~Malle, Tables of cyclic cubic extensions $K/\mathbb{Q}$ with $|d_K| < 10^{16}$: \url{https://service.mathematik.uni-kl.de/~numberfieldtables/KT_3/download.html}.
\bibitem{MA} G. ~Malle, M.~Adam {\em A class group heuristic based on the distribution of $1$-eigenspaces in matrix groups}, J. Number Theory {\bf 149} (2015), 225--235.
\bibitem{MAD} D.C.~Mayer, M.~Ayadi, A.~Derhem, {\em $p$-class field towers of cyclic cubic fields}, preprint, 2018.
\bibitem{O} E.A.~O'Brien, {\em The p-group generation algorithm}, J. Symbolic Comput. {\bf 9} (1990), 677--698.
\bibitem{O2} E.A.~O'Brien, {\em Isomorphism testing for p-groups}, J. Symbolic Comput. {\bf 17} (1994), 133--147.
\bibitem{pari} The PARI~Group, PARI/GP version {\tt 2.9.2}, Bordeaux, 2017, \url{http://pari.math.u-bordeaux.fr/}.
\bibitem{RS} S.~Rubinstein-Salzedo, {\em Invariants for A4 fields and the Cohen–Lenstra heuristics}, International J. Number Theory {\bf 10(5)} (2014), 1259--1276.
\bibitem{Se} J.-P.~Serre, {\em Topics in Galois Theory}, Taylor \& Francis (2008).

\end{thebibliography}
\end{document}